\documentclass[12pt,reqno]{amsart}
\usepackage{amssymb,amsfonts,amsthm}
\usepackage{ulem}
\usepackage{amsthm}
\usepackage{amssymb}
\usepackage{mathrsfs}
\usepackage{frcursive}
\usepackage[english,american]{babel}
\usepackage{enumerate}
\usepackage{graphicx}
\usepackage{tikz} 
\usepackage{setspace}
\usepackage{hyperref}
\usepackage{color}

\definecolor{Red}{cmyk}{0,1,1,0.2}

\usepackage{amsfonts,amsmath,layout}
\newcommand{\R}{\mathbb R}

\def\R{\mathbb R}

\newcommand{\be}{\begin{equation}}
\newcommand{\ee}{\end{equation}}

\def\1{{\bf 1}}

\def\ds{\displaystyle}
\newtheorem{Theorem}{Theorem}[section]
\newtheorem{Definition}[Theorem]{Definition}
\newtheorem{Proposition}[Theorem]{Proposition}
\newtheorem{Lemma}[Theorem]{Lemma}

\newtheorem{Remark}[Theorem]{Remark}

\begin{document}
\title[Fully Nonlinear PDEs on Star Networks]{Well-Posedness of Fully Nonlinear PDEs on Star-Shaped Networks with Nonlinear Kirchhoff Conditions}
\space
\space
\author[Isaac Ohavi]{Isaac Ohavi $^\aleph$
\\
$^\aleph$Kyiv School of Economics, Computer Science Department, Mathematics, Kyiv, Ukraine}
\address{ Kyiv School of Economics, Computer Science Department, Mathematics, Kyiv, Ukraine}
\email{iohavi@kse.org.ua \& isaac.ohavi@gmail.com}
\dedicatory{Version: \today}
\maketitle
\begin{abstract} 
The aim of this paper is to establish the well-posedness of fully nonlinear partial differential equations on star-shaped networks subject to nonlinear Kirchhoff conditions at the junction. The branchwise operators may be discontinuous across the junction, and uniform ellipticity is required only in a neighborhood of the vertex. Our approach relies on the construction of auxiliary test functions solving suitable nonlinear ordinary differential equations, whose structure allows us to control the viscosity inequalities at the junction. We first show that generalized Kirchhoff viscosity solutions in the sense of Lions--Souganidis \cite{Lions Souganidis 1,Lions Souganidis 2} satisfy the strong Kirchhoff condition. Combining these auxiliary functions with an affine perturbation of supersolutions and precise junction gradient estimates, we establish a comparison principle for bounded semicontinuous viscosity sub- and supersolutions. Existence and uniqueness are then addressed through Perron's method and the comparison principle. Finally, we discuss first-order junction problems and a possible transmission condition arising from vanishing viscosity.
\end{abstract}
{\small \textbf{AMS Class:} 35J66, 35D40, 35B51, 35F21, 35J70.}
\section{Introduction}

This work is a first step toward a systematic study of
fully nonlinear partial differential equations (PDEs)
on networks, with potential applications to traffic flow,
optimal control, and diffusion processes on graphs. Let $I\geq 2$ be an integer, and consider a compact
star-shaped network
\[
\mathcal N_R
=
\left(
\bigcup_{i=1}^{I}
([0,R]\times\{i\})
\right)\big/\sim,
\qquad
(0,i)\sim(0,j),\quad \forall i,j\in[I],
\]
consisting of $I$ branches
$\mathcal{R}_i\cong[0,R]$, with $R>0$, emanating
from a common junction point $\mathbf{0}$.
Throughout this paper, we use the notation
$[I]:=\{1,\ldots,I\}$. We consider the following system:
\label{sec intro}
\begin{eqnarray}\label{eq PDE 0}
&\begin{cases}
\textbf{Fully nonlinear second-order equations on each ray :}\\ 
Q_i\big(x,u_i(x),\partial_xu_i(x),\partial_x^2u_i(x)\big)=0,~~x\in(0,R),\\
\textbf{Nonlinear Kirchhoff boundary condition at }\bf 0:\\
F\big(u(0),\partial_xu_1(0),\dots,\partial_xu(0)\big)=0,\\
\textbf{Dirichlet boundary condition at }x=R:\\
u_i(R)=a_i,~~\forall i\in [I].
\end{cases}
\end{eqnarray}
For simplicity, we assume throughout this paper that all
branches $\mathcal{R}_i=[0,R]\times\{i\}$, $i\in[I]$,
have the same length $R>0$.
However, our approach could potentially be extended to more general networks,
including those involving a combination of nonlinear
Kirchhoff and Dirichlet boundary conditions at different
vertices.

The main objective of this paper is to establish the
well-posedness of system \eqref{eq PDE 0} in a fully
nonlinear second-order framework, where the branchwise
operators are assumed to be uniformly elliptic only
in a neighborhood of the junction. In particular,
our assumptions allow for degeneracy away from the
vertex and discontinuities of the Hamiltonians across
different branches.

We work in the framework of bounded semicontinuous
viscosity solutions, namely lower semicontinuous
supersolutions and upper semicontinuous subsolutions.
Our main contributions are threefold:
\begin{itemize}
    \item We prove that every generalized Kirchhoff
    viscosity subsolution or supersolution, in the
    sense of Lions--Souganidis
    \cite{Lions Souganidis 1,Lions Souganidis 2},
    satisfies the corresponding strong Kirchhoff
    condition at the junction.

    \item We establish a comparison principle for
    bounded semicontinuous viscosity subsolutions
    and supersolutions.

    \item We investigate existence and uniqueness
    of continuous viscosity solutions through
    Perron's method and the comparison principle.
\end{itemize}
Finally, we discuss the first-order counterpart of
the junction problem, including the role of
non-coercive Hamiltonians and a possible transmission
condition arising from vanishing viscosity.

Our approach is partly inspired by the construction of
auxiliary test functions introduced in \cite{Ohavi visco}
for uniformly elliptic second-order HJB equations
on star-shaped networks. In the present work, we develop
a new family of test functions for supersolutions
(resp. subsolutions), denoted by
$\overline{\phi}^{\mathbf{p}}$
(resp. $\underline{\phi}^{\mathbf{p}}$),
as solutions of suitable nonlinear second-order
ordinary differential equations on each branch. These auxiliary equations involve the elliptic term
$-\sigma\phi''$, the zero-order term $\lambda\phi$,
and first-order nonlinear terms of the form
$\pm C_M(1+|\phi'|)$, which control the contributions
of the branchwise Hamiltonians. Together with
appropriately chosen boundary conditions, these
equations yield convex and concave auxiliary
functions whose contact properties and junction
gradients play a central role in our analysis.

The auxiliary test functions are constructed so that
their contact with the viscosity sub- and supersolutions
occurs at the junction. Moreover, the corresponding
branchwise Hamiltonians are strictly positive for
subsolutions and strictly negative for supersolutions
at the vertex. Consequently, the generalized junction
conditions introduced by Lions--Souganidis
\cite{Lions Souganidis 1,Lions Souganidis 2}
necessarily reduce to the strong Kirchhoff inequalities
for these auxiliary test functions.

To establish this property for arbitrary admissible
test functions, we exploit the estimates on the
second derivatives of the auxiliary solutions.
These estimates allow us to approximate the junction
gradients of any admissible test function and,
by continuity of the Kirchhoff operator, to recover
the corresponding strong boundary inequality.
We thus prove that every generalized Kirchhoff
viscosity subsolution or supersolution satisfies
the strong Kirchhoff condition.

An important consequence of this equivalence is
that the junction condition can be formulated
independently of the values of the branchwise
Hamiltonians at the vertex. This is particularly
useful when these Hamiltonians are discontinuous
across the junction, a setting in which the classical
doubling-of-variables technique cannot be applied
directly across different branches.

We then establish a comparison principle for bounded
semicontinuous viscosity sub- and supersolutions,
using once again the auxiliary test functions
$\overline{\phi}^{\mathbf{p}}$ and
$\underline{\phi}^{\mathbf{p}}$.
The key argument combines precise estimates of their
junction gradients with a suitable affine perturbation
of the supersolution, which introduces the strictness
needed in the Kirchhoff inequality.
Finally, existence is addressed through Perron's method,
while uniqueness follows from the comparison principle.

Our analysis is entirely based on PDE techniques
and does not rely on convexity assumptions on the
Hamiltonians or on any control-theoretic representation.
The main results are established in the fully nonlinear
second-order setting, under a uniform ellipticity
assumption restricted to a neighborhood of the junction.\\
In the final section, we discuss the corresponding
first-order junction problems, highlighting some
difficulties arising in the absence of coercivity
and examining a possible transmission condition
selected by vanishing viscosity.\\
Moreover, we expect that the techniques developed here
could also be adapted to parabolic problems on networks,
although such an extension is beyond the scope of
the present work.\\
We now review some of the main contributions in the
literature concerning PDEs on networks and junction
problems closely related to system \eqref{eq PDE 0}.

We first mention \cite{Ohavi visco}, which provides
part of the motivation for the auxiliary test-function
approach developed in the present paper. That work
investigates uniformly elliptic second-order
Hamilton-Jacobi-Bellman (HJB) equations on star-shaped
networks, subject to nonlinear local-time Kirchhoff
boundary conditions. Its analytical framework
introduces an additional local-time variable at
the junction, together with test functions involving
derivatives with respect to this variable, designed
to compensate for the interactions between the
branchwise Hamiltonians and the junction condition.
In contrast, the present work develops auxiliary
ordinary differential equations directly on the
branches, without introducing an additional
local-time variable.

From a probabilistic perspective, the nonlinear
local-time Kirchhoff condition is closely related
to Walsh spider diffusions whose spinning measures
depend on their own local time at the junction,
as investigated in \cite{Spider,Spider 2}.
This connection motivates the stochastic scattering
control problem formulated in \cite{Ohavi control},
where the controlled selection of the diffraction
probabilities at the vertex leads to a Hamilton--Jacobi--Bellman system coupled through a nonlinear local-time
Kirchhoff condition. We also refer to
\cite{linear PDE} for the analysis of linear parabolic
equations on star-shaped networks involving
local-time-dependent Kirchhoff boundary conditions.

Finally, we emphasize that the Kirchhoff condition
in \cite{Ohavi visco} is formulated in the strong
viscosity sense, so that its definition does not
require evaluating the branchwise Hamiltonians
at the junction.

For a comprehensive account of recent developments
in Hamilton-Jacobi equations and optimal control
problems involving discontinuities, we refer to the
monograph \cite{Barles book}, particularly Part III,
which is devoted to problems on networks and
stratified domains.

A substantial literature has also been devoted to
linear, quasilinear, and semilinear equations on
networks, including problems involving Kirchhoff-type
conditions and mean field games. We refer, among
others, to
\cite{Ohavi PDE,Camilli 1,Camilli 2,Achdou 1,
Achdou 2,Barles semi linear,Below 3}.

An important contribution to the viscosity theory
of junction problems was made by Lions and
Souganidis. In \cite{Lions Souganidis 1}, they
investigate first-order Hamilton-Jacobi equations
on one-dimensional junctions with nonconvex
coercive Hamiltonians, introducing a generalized
notion of viscosity solutions and establishing
well-posedness and stability results.
In \cite{Lions Souganidis 2}, they extend their
analysis to multidimensional junction problems
with Kirchhoff-type conditions, including
first- and second-order equations.

More recently, Forcadel and Monneau
\cite{Forcadel Monneau 2026} investigated
relaxed junction conditions for time-dependent
Hamilton-Jacobi equations on networks with
coercive Hamiltonians. They showed that viscosity
solutions satisfying a prescribed junction condition
in a weak sense also satisfy a strong effective
junction condition, characterized through a
suitable relaxation operator.

Although their framework is essentially first-order,
their results highlight the importance of the
relationship between weak and strong formulations
of junction conditions. In the present work,
we investigate a related question in a different
setting, involving fully nonlinear second-order
operators that are uniformly elliptic near the
junction, without imposing coercivity or convexity
assumptions on the Hamiltonians.

The works of Lions and Souganidis also provide
an important framework for generalized junction
conditions and their relationship with flux-limited
formulations. Related approaches have been developed
by Imbert--Monneau and Imbert--Nguyen
\cite{control 5,control 6}.

%In the present paper, we investigate the relationship
%between generalized and strong Kirchhoff conditions
%in a different setting, involving fully nonlinear
%second-order operators that are uniformly elliptic
%near the junction and may be discontinuous
%across different branches.

For deterministic optimal control approaches to
Hamilton-Jacobi equations on networks and
stratified domains, we refer to
\cite{control 1,control 2,control 3,control 4,
control 5,control 6}.
Related developments concerning conservation laws,
traffic flow, and Hamilton-Jacobi equations
on junctions can be found in
\cite{Carda junction -1,Carda junction 0,
Carda junction 1,Carda junction 2}
and the references therein.

Finally, a substantial body of literature is devoted
to comparison principles and existence results for
viscosity solutions of second-order elliptic and
parabolic PDEs subject to nonlinear Neumann-type
and oblique derivative boundary conditions on
classical domains. We refer, among others, to
\cite{Barles Neuman 1,Barles Neuman 2,Daniel game,
Hishi neumann,Lions neumann,Sato}
and the references therein.

The paper is organized as follows.
In Section~\ref{sec : Notations}, we introduce the
notation and assumptions and state our three main
results: the equivalence between generalized and
strong Kirchhoff conditions
(Theorem~\ref{th : weak are strong}),
the comparison principle
(Theorem~\ref{th : compa syst}),
and the well-posedness result
(Theorem~\ref{th : existence}).\\
Section~\ref{sec preuves} is devoted to the proofs of these results
and is divided into four subsections.
In Subsection~\ref{subs ODE}, we construct auxiliary test
functions by solving suitable nonlinear ordinary
differential equations and establish their
fundamental estimates and contact properties
at the junction.
In Subsection~\ref{subsec weak fort}, we prove that generalized
Kirchhoff viscosity solutions satisfy the
strong Kirchhoff condition.
In Subsection~\ref{subs comparison}, we establish the comparison principle by combining an affine perturbation
of supersolutions with precise estimates of
the junction gradients of the auxiliary functions.
We then obtain existence through Perron's method
and uniqueness by comparison in Subsection~\ref{subs existence Perron}.\\
Finally, Section~\ref{se 1 ordre et sto} discusses first-order junction
problems, including difficulties arising in the
non-coercive setting and a possible transmission
condition selected by vanishing viscosity.

\section{Notation, Definitions and Main Results} \label{sec : Notations}
Let \(I\geq 2\) be an integer, and consider a compact star-shaped network
\[
\mathcal N_R
=
\left(
\bigcup_{i=1}^{I}
([0,R]\times\{i\})
\right)\big/\sim,
\qquad
(0,i)\sim(0,j),\quad \forall i,j\in[I].
\]
consisting of \(I\) branches \(\mathcal R_i\simeq[0,R]\), with \(R>0\), sharing a common junction point denoted by \(0\). Each point of \(\mathcal N_R\) is represented by a pair \((x,i)\), where \(x\in[0,R]\) and \(i\in[I]\), with all points \((0,i)\), \(i\in[I]\), identified with the common junction \(0\).\\
We introduce the following data:
$$\textbf{Data:~} (\mathcal{D})~~\begin{cases}
\Big(Q_i \in \mathcal{C}\big([0,R]\times \R^3,\R\big)\Big)_{i\in[I]}\\
F\in \mathcal{C}\big(\R\times \R^I,\R\big)\\
\big(a_i\big)_{i\in[I]}\in \R^I,
\end{cases}.
$$
We assume that the data $(\mathcal{D})$ satisfy the following assumption:
$$\textbf{Assumption } (\mathcal{H})$$
a) The growth assumption for the nonlinear terms $(Q_i)_{i\in[I]}$ with respect to their second variable: there exists a strictly positive constant $\lambda>0$, independent of $i\in[I]$,
such that:
\begin{eqnarray*}
& \forall (x,u,v,p,X)\in [0,R]\times \R^4,~~\text{if}~~v\leq u,~\text{then}~~\forall i\in[I]:\\
&Q_i(x,u,p,X)-Q_i(x,v,p,X)\ge \lambda(u-v).
\end{eqnarray*}
b) We recall the classical structural condition
compatible with Ishii's matrix inequality,
as introduced in \cite{User guide}, for the nonlinear terms $(Q_i)_{i\in[I]}$, that reads:
\begin{eqnarray*}
 &\exists \omega_Q \in  \mathcal{C}(\R_+,\R_+),~~\omega_Q(0)=0,~~
 \forall \alpha>0,~~\forall (x,y,u,X,Y)\in [0,R]^2\times \R^3,~\text{such that:}\\
 &-3\alpha \begin{pmatrix}1&0\\
0&1
 \end{pmatrix}\leq\begin{pmatrix}X&0\\
 0&-Y
 \end{pmatrix}\leq 3\alpha \begin{pmatrix}1&-1\\
-1&1
 \end{pmatrix},~~\text{then}~~\forall i\in[I]:\\
&Q_{i}(y,u,\alpha(x-y),Y)-Q_{i}(x,u,\alpha(x-y),X)\leq \omega_Q\big(\alpha|x-y|^2+|x-y|\big).
\end{eqnarray*}
c) The nonlinear terms $(Q_i)_{i\in[I]}$ are uniformly elliptic 
near the junction: there exist $\varepsilon_0\in (0,R]$ and 
$\sigma>0$ such that, for every $i\in[I]$, 
$x\in[0,\varepsilon_0]$, and $(u,p,X,Y)\in\mathbb{R}^4$ with $X\ge Y$:
\[
Q_i(x,u,p,Y)
\ge
Q_i(x,u,p,X)
+
\sigma(X-Y).
\]
d) Local Lipschitz continuity with respect to the gradient:
for every $M>0$, there exists a constant $L_M>0$,
independent of $i\in[I]$, such that, for all
$x\in[0,R]$, $r\in[-M,M]$, and $(p,q,X)\in \R^3$: \[
|Q_i(x,r,p,X)-Q_i(x,r,q,X)|
\leq L_M|p-q|.
\]
e) The nonlinear Kirchhoff term $F$ is nonincreasing with respect to its first variable:
\begin{align*}
&\forall (u,v,p_1,\ldots p_I)\in \R^2 \times \R^I,~\text{if}~u\ge v,~~\text{then:}~~F(u,p_1,\ldots p_I)\leq F(v,p_1,\ldots p_I).
\end{align*}
f) The nonlinear Kirchhoff term $F$ is uniformly strictly increasing with respect to its gradient, namely:
\begin{align*}
\exists \alpha >0,~~\forall (u,p_1,\ldots p_I,q_1,\ldots,q_I)\in \R\times \R^I\times \R^I,~~\text{s.t.}~~\text{if:}~~\forall i\in[I]~~p_i\ge q_i,~\text{then:}&\\
F(u,p_1,\ldots p_I)-F(u,q_1,\ldots q_I)\ge \alpha \big(\sum_{i=1}^I p_i-q_i\big).&
\end{align*}
g) Local uniform continuity of the nonlinear Kirchhoff
operator with respect to the gradient variables:
for every $M>0$, there exists a modulus of continuity
$\omega_{F,M}\in\mathcal{C}(\R_+,\R_+)$,
with $\omega_{F,M}(0)=0$, such that,
for all $u\in[-M,M]$ and $p,q\in\R^I$:
\[
|F(u,p)-F(u,q)|
\leq \omega_{F,M}(|p-q|_{\R^I}).
\]
Next, we introduce the set of lower semi-continuous $\mathcal{L}_{sc}(\mathcal{N}_R,\R)$ and upper semi-continuous $\mathcal{U}_{sc}(\mathcal{N}_R,\R)$ functions defined on the star-shaped network $\mathcal{N}_R$, and valued in $\R$: 
\begin{align}\label{eq semi conti 0}
&\nonumber\mathcal{L}_{sc}(\mathcal{N}_R,\R):=\Big\{~~f:\mathcal{N}_R\to \R,~~(x,i)\mapsto f_i(x)~~\text{ s.t}~~\forall i\in [I]~~x\mapsto f_i(x)~~\text{is lower}\\
&\nonumber\hspace{3cm}\text{semi-continuous in}~ [0,R],~~\text{and}~~f({\bf 0})=\underset{j\in [I]}{\text{min}}f_j(0)~~\Big\},\\
&\nonumber\mathcal{U}_{sc}(\mathcal{N}_R,\R):=\Big\{~~f:\mathcal{N}_R\to \R,~~(x,i)\mapsto f_i(x)~~\text{ s.t}~~\forall i\in [I]~~x\mapsto f_i(x)~~\text{is upper} \\
&\hspace{3cm}\text{semi-continuous in}~ [0,R],~~\text{and}~~f({\bf 0})=\underset{j\in [I]}{\text{max}}f_j(0)~~\Big\}.
\end{align}
Here, the values \(f_i(0)\), \(i\in[I]\), refer to the values of the branchwise semicontinuous functions before identification at the junction. The above minimum and maximum conventions assign a unique value \(f(0)\) after gluing the branches together.\\
It is straightforward to verify that if a function is lower semicontinuous on each branch, then assigning at the junction the minimum of its branchwise value at the junction point yields a lower semicontinuous function on the whole network endowed with the geodesic metric. The corresponding statement for upper semicontinuity follows by taking the maximum. Recall that the geodesic metric $d^\mathcal{N_R}$ of the star-shaped network is defined by:
\begin{equation*}
\forall \Big((x,i),(y,j)\Big)\in \mathcal{N}_R^2,~~d^\mathcal{N_R}\Big((x,i),(y,j) \Big)  := \left\{
\begin{array}{ccc}
 |x-y| & \mbox{if }  & i=j\;,\\ 
x+y & \mbox{if }  & {i\neq j}.\;
\end{array}\right.
\end{equation*}
We also denote $\mathcal{L}_{sc}^b(\mathcal{N}_R,\R)$ (resp. $\mathcal{U}_{sc}^b(\mathcal{N}_R,\R)$) the subsets of $\mathcal{L}_{sc}(\mathcal{N}_R,\R)$ (resp. $\mathcal{U}_{sc}(\mathcal{N}_R,\R)$) of bounded semi-continuous maps defined on $\mathcal{N}_R$. Naturally, we also introduce $\forall k\in \mathbf{N}$:
\begin{align*}
&\mathcal{C}^{k}(\mathcal{N}_R,\R):=\Big\{~~f:\mathcal{N}_R\to \R,~~(x,i)\mapsto f_i(x)~~\text{ s.t}~~\forall i\in [I]~~x\mapsto f_i(x)~~\text{is}~~\mathcal{C}^{k}([0,R],\R)\\
&\hspace{3cm}\text{and}~~f({\bf 0})=f_j(0)=f(0)~~\forall (i,j)\in [I]^2~~\Big\}.
\end{align*}
We now introduce generalized viscosity Kirchhoff subsolutions and supersolutions in the sense of Lions–Souganidis \cite{Lions Souganidis 1, Lions Souganidis 2}, together with their strong counterparts, as considered in \cite{Ohavi visco}.
\begin{Definition} \label{def : weak sur/sous solutions}\text{Generalized viscosity Kirchhoff solution: Lions-Souganidis \cite{Lions Souganidis 1}-\cite{Lions Souganidis 2}}:\\
a) We say that $u\in \mathcal{L}_{sc}^b(\mathcal{N}_R,\R)$ is generalized viscosity Kirchhoff supersolution of
\eqref{eq PDE 0}, if for every test function $\phi \in \mathcal{C}^{2}\big(\mathcal{N}_R,\R\big)$, and at every local minimum point $(x_\star,i_\star)\in \mathcal{N}_R$ of $u-\phi$, with $u_{i_\star}(x_\star)-\phi_{i_\star}(x_\star)=0$, we have:
\begin{eqnarray*}
\begin{cases}
Q_{i_\star}\big(x_\star,u_{i_\star}(x_\star),\partial_x\phi_{i_\star}(x_\star),\partial_x^2\phi_{i_\star}(x_\star)\big)\ge 0,~\text{if}~x_\star \in (0,R),\\
\max\Big\{\underset{i\in[I]}{\max}\{Q_{i}\big(0,u(0),\partial_x\phi_{i}(0),\partial_x^2\phi_{i}(0)\big)\},-F\big(u(0),\partial_x \phi_1(0),\ldots,\partial_x \phi(0)\big)\Big\}\ge 0,\\
\text{if}~x_\star = \bf 0.
\end{cases}
\end{eqnarray*}
b) We say that $u\in \mathcal{U}_{sc}^b(\mathcal{N}_R,\R)$ is a generalized viscosity Kirchhoff subsolution of
\eqref{eq PDE 0}, if for every test function $\phi \in \mathcal{C}^{2}\big(\mathcal{N}_R,\R\big)$, and at every local maximum point $(x_\star,i_\star)\in \mathcal{N}_R$ of $u-\phi$, with $u_{i_\star}(x_\star)-\phi_{i_\star}(x_\star)=0$, we have:
\begin{eqnarray*}
\begin{cases}
Q_{i_\star}\big(x_\star,u_{i_\star}(x_\star),\partial_x\phi_{i_\star}(x_\star),\partial_x^2\phi_{i_\star}(x_\star)\big)\leq 0,~\text{if}~x_\star \in (0,R),\\
\min\Big\{\underset{i\in[I]}{\min}\{Q_{i}\big(0,u(0),\partial_x\phi_{i}(0),\partial_x^2\phi_{i}(0)\big)\},-F\big(u(0),\partial_x \phi_1(0),\ldots,\partial_x \phi(0)\big)\Big\}\leq 0, \\
\text{if}~x_\star =\bf 0.
\end{cases}
\end{eqnarray*}
c) A generalized continuous viscosity Kirchhoff solution is a continuous function on $\mathcal{N}_R$ that is simultaneously a generalized viscosity Kirchhoff sub and supersolution. 
\end{Definition}
\begin{Definition} \label{def : strong sur/sous solutions}\text{Viscosity Kirchhoff solution}:\\
a) We say that $u\in \mathcal{L}_{sc}^b(\mathcal{N}_R,\R)$ is a viscosity Kirchhoff supersolution of
\eqref{eq PDE 0}, if for every test function $\phi \in \mathcal{C}^{2}\big(\mathcal{N}_R,\R\big)$, and at every local minimum point $(x_\star,i_\star)\in \mathcal{N}_R$ of $u-\phi$, with $u_{i_\star}(x_\star)-\phi_{i_\star}(x_\star)=0$, we have:
\begin{eqnarray*}
\begin{cases}
Q_{i_\star}\big(x_\star,u_{i_\star}(x_\star),\partial_x\phi_{i_\star}(x_\star),\partial_x^2\phi_{i_\star}(x_\star)\big)\ge 0,~\text{if}~x_\star \in (0,R),\\
F\big(u(0),\partial_x \phi_1(0),\ldots,\partial_x \phi(0)\big)\leq 0, ~~\text{if}~x_\star = \bf 0.
\end{cases}
\end{eqnarray*}
b) We say that $u\in \mathcal{U}_{sc}^b(\mathcal{N}_R,\R)$ is a viscosity Kirchhoff subsolution of
\eqref{eq PDE 0}, if for every test function $\phi \in \mathcal{C}^{2}\big(\mathcal{N}_R,\R\big)$, and at every local maximum point $(x_\star,i_\star)\in \mathcal{N}_R$ of $u-\phi$, with $u_{i_\star}(x_\star)-\phi_{i_\star}(x_\star)=0$, we have:
\begin{eqnarray*}
\begin{cases}
Q_{i_\star}\big(x_\star,u_{i_\star}(x_\star),\partial_x\phi_{i_\star}(x_\star),\partial_x^2\phi_{i_\star}(x_\star)\big)\leq 0,~\text{if}~x_\star \in (0,R),\\
F\big(u(0),\partial_x \phi_1(0),\ldots,\partial_x \phi(0)\big)\ge 0, ~~\text{if}~x_\star = \bf 0.
\end{cases}
\end{eqnarray*}
c) A continuous viscosity  Kirchhoff solution is a continuous function on $\mathcal{N}_R$ that is simultaneously a viscosity Kirchhoff sub and supersolution in the sense above. 
\end{Definition}
Our first main result establishes that every generalized viscosity Kirchhoff subsolution or supersolution in the sense of Definition \ref{def : weak sur/sous solutions} satisfies the corresponding strong Kirchhoff condition of Definition \ref{def : strong sur/sous solutions}.
\begin{Theorem}\label{th : weak are strong}
Assume that assumption $(\mathcal{H})$ holds. Let $u \in \mathcal{L}_{sc}^b(\mathcal{N}_R,\R)$ (resp. $v \in \mathcal{U}_{sc}^b(\mathcal{N}_R,\R)$ be a generalized viscosity Kirchhoff super (resp. sub) solution in the sense of Definition \ref{def : weak sur/sous solutions} of system \eqref{eq PDE 0}. Then $u$ (resp. $v$) is a viscosity Kirchhoff super (resp. sub) solution, in the sense of Definition \ref{def : strong sur/sous solutions} of system \eqref{eq PDE 0}.   
\end{Theorem}
We next establish a comparison principle for bounded semicontinuous viscosity subsolutions and supersolutions, without introducing an additional junction local-time variable, as in \cite{Ohavi visco}. Our framework allows the branchwise equations to be degenerate away from the junction.
\begin{Theorem} (Comparison Theorem)\label{th : compa syst}
Assume that assumption $(\mathcal{H})$ holds. Let $u \in \mathcal{L}_{sc}^b(\mathcal{N}_R,\R)$ be a Kirchhoff supersolution and  $v \in \mathcal{U}_{sc}^b(\mathcal{N}_R,\R)$ be a Kirchhoff subsolution of \eqref{eq PDE 0}. If for all $i\in [I]$, $u_i(R)\ge v_i(R)$, then:
$$\forall (x,i)\in \mathcal{N}_R,~~u_i(x)\ge v_i(x). $$
\end{Theorem}
Finally, combining Perron's method with the preceding results, we obtain the following existence and uniqueness theorem.
\begin{Theorem} (Well-posedness) \label{th : existence}
Assume that assumption $(\mathcal{H})$ holds. Suppose, in addition, that there exist constants $\delta_R \in (0,R)$ and $\sigma_R>0$ such that, for every $i\in[I]$, $x\in[R-\delta_R,R]$, $(r,p,X)\in\R^3$ and $t\geq 0$,
\[
Q_i(x,r,p,X)-Q_i(x,r,p,X+t)\geq \sigma_R t.
\]
Then there exists a unique continuous viscosity Kirchhoff solution $u\in\mathcal{C}(\mathcal{N}_R,\R)$ of system \eqref{eq PDE 0}.
\end{Theorem}
In the last statement, the additional uniform ellipticity assumption near the outer endpoints is used only to construct local Dirichlet barriers ensuring the attainment of the prescribed boundary data in Perron's method. It is not needed for the comparison principle. Indeed, for degenerate elliptic equations, Dirichlet boundary conditions need not be attained in the classical sense; see \cite[Section 7.C]{User guide}.
\section{Proof of the main results}\label{sec preuves}
In this section, we establish the three main results of the paper: the equivalence between the generalized and strong Kirchhoff conditions, the comparison principle, and the existence of viscosity solutions via Perron's method. The proofs are organized into four subsections, corresponding to the construction of auxiliary test functions, the equivalence between generalized and strong Kirchhoff conditions, the comparison principle, and the existence and uniqueness result.\\
We first recall that the system under consideration is given by
\begin{eqnarray}\label{eq : system second order}
\begin{cases}
Q_i\big(x,h_i(x),\partial_xh_i(x),\partial_x^2h_i(x)\big)=0,~~x\in  (0,R),\\
F\big(h(0),\partial_xh_1(0),\ldots,\partial_xh(0)\big)=0,~~x=\bf 0,\\
h_i(R)=a_i,~~\forall i\in [I].
\end{cases}.
\end{eqnarray}
\subsection{Construction of Envelope Test Functions with Extremal Contacts at the Junction}\label{subs ODE}
We construct a family of auxiliary test functions as solutions to suitable ordinary differential equations incorporating the errors arising from the Hamiltonian envelopes. The signs of the terms in these equations are carefully chosen to exclude interior extrema, while the boundary conditions are designed to force the relevant extrema to occur at the junction. They provide local test functions near the junction which will allow us to prove that, under Assumption $(\mathcal{H})$, every generalized Kirchhoff viscosity solution in the sense of Definition \ref{def : weak sur/sous solutions} satisfies the strong Kirchhoff condition of Definition \ref{def : strong sur/sous solutions}, as stated in Theorem \ref{th : weak are strong}. Second, the same construction provides the local test functions needed later in the proof of the comparison theorem, Theorem \ref{th : compa syst}.\\
We first recall the constants entering the construction. The constant $\lambda>0$ is the uniform strict monotonicity constant with respect to the zero-order variable appearing in Assumption $(\mathcal{H})$ (a). Moreover, by the uniform ellipticity assumption near the junction, namely Assumption $(\mathcal{H})$ (c), there exist $\varepsilon_0>0$ and $\sigma>0$ such that the corresponding ellipticity estimate holds for every $i\in[I]$ and every $x\in[0,\varepsilon_0]$.\\
Let $u$ and $v$ denote, respectively, the viscosity supersolution and subsolution under consideration.
Since $u$ and $v$ are bounded, we fix once and for all a constant $M>0$, independent of the parameters introduced below, such that:
\begin{eqnarray}\label{eq strict born}
 M>
\max\left\{
\sup_{i\in[I]}\sup_{x\in[0,\varepsilon_0]}|u_i(x)|,
\,
\sup_{i\in[I]}\sup_{x\in[0,\varepsilon_0]}|v_i(x)|
\right\}.   
\end{eqnarray}
For this fixed value of $M$, $(\mathcal{H})$ (d) provides
a constant $L_M>0$, independent of $i\in[I]$.
Since the functions $(Q_i)_{i\in[I]}$ are continuous,
we may define:
\[
K_M:=\max_{i\in[I]}\sup_{\substack{x\in[0,R]\\ |r|\le M}}
|Q_i(x,r,0,0)|<\infty,
\qquad C_M:=\max\{K_M,L_M\}.
\]
Consequently, for every $i\in[I]$, $x\in[0,R]$,
$|r|\le M$ and $p\in\mathbb{R}$, we have:
\begin{eqnarray}\label{eq def CM}
\nonumber |Q_i(x,r,p,0)|
&\le |Q_i(x,r,0,0)|
+|Q_i(x,r,p,0)-Q_i(x,r,0,0)|\\
&\le K_M+L_M|p|\le C_M(1+|p|). 
\end{eqnarray}
Throughout the construction, we choose:
\[
0<\varepsilon<\varepsilon_0,
\qquad
0<\theta,
\qquad
\eta>0.
\]
For notational convenience, throughout this section we collect the
parameters involved in the construction into the vector:
\[
\mathbf{p}:=(\varepsilon,\theta,\eta).
\]
Each component of $\mathbf{p}$ has a specific role.
The parameter $\varepsilon$ determines the size of the neighborhood:
\[
\mathcal V_\varepsilon
=
\left(
\bigcup_{i=1}^{I}
([0,\varepsilon]\times\{i\})
\right)\big/\sim,
\qquad
(0,i)\sim(0,j),\quad \forall i,j\in[I].
\]
of the junction point. The parameter $\theta$ introduces opposite perturbations of the boundary values at $x=\varepsilon$: the test function associated with the supersolution $u$ is shifted downward by $-\theta$, whereas the one associated with the subsolution $v$ is shifted upward by $+\theta$, while their values at the junction remain unchanged. 
We choose in the sequel $\eta>0$ sufficiently large so that:
\begin{eqnarray*}
\eta>\lambda(M+\theta).    
\end{eqnarray*}
This choice will ensure the convexity and concavity
properties of the auxiliary test functions constructed
below. The parameter $\eta$ introduces the strictness required in the differential relations satisfied by the auxiliary functions and yields the strict sign needed in the viscosity inequalities involving the nonlinear terms $(Q_i)_{i\in[I]}$.\\
Given the viscosity supersolution $u$, we introduce $\overline{\phi}^{\mathbf{p}}$, whose components solve the following system:
\begin{equation}
\label{eq : system EDO test sur ordre 2}
\left\{
\begin{alignedat}{2}
&\lambda\overline{\phi}^{\mathbf{p}}_i(x)
-\sigma\partial_x^2\overline{\phi}^{\mathbf{p}}_i(x)
+C_M\Big(1+
\big|\partial_x\overline{\phi}^{\mathbf{p}}_i(x)\big|
\Big)
=-\eta,
&\qquad &x\in(0,\varepsilon),
\\[2mm]
&\overline{\phi}^{\mathbf{p}}(0)=u(0),
\\[2mm]
&\overline{\phi}^{\mathbf{p}}_i(\varepsilon)
=u_i(\varepsilon)-\theta,
&\qquad &i\in[I].
\end{alignedat}
\right.
\end{equation}
Similarly, given the viscosity subsolution $v$, we introduce
$\overline{\phi}^{\mathbf{p}}$, whose components solve:
\begin{equation}
\label{eq : system EDO test sous ordre 2}
\left\{
\begin{alignedat}{2}
&\lambda\underline{\phi}^{\mathbf{p}}_i(x)
-\sigma\partial_x^2\underline{\phi}^{\mathbf{p}}_i(x)
-C_M\Big(1+
\big|\partial_x\underline{\phi}^{\mathbf{p}}_i(x)\big|
\Big)
=\eta,
&\qquad &x\in(0,\varepsilon),\ i\in[I],
\\[2mm]
&\underline{\phi}^{\mathbf{p}}(0)=v(0),
\\[2mm]
&\underline{\phi}^{\mathbf{p}}_i(\varepsilon)
=v_i(\varepsilon)+\theta,
&\qquad &i\in[I].
\end{alignedat}
\right.
\end{equation}
The following lemma establishes the existence, uniqueness, and fundamental estimates for the auxiliary test functions above, which will be essential in the subsequent analysis.
\begin{Lemma}\label{lm estimee importantes}
Let $0<\varepsilon<\varepsilon_0$, $\theta>0$, and
$\eta>0$ be such that:
\[
\eta>\lambda(M+\theta),
\qquad
\mathbf{p}:=(\varepsilon,\theta,\eta).
\]
Then the auxiliary systems
\eqref{eq : system EDO test sur ordre 2} and
\eqref{eq : system EDO test sous ordre 2}
admit unique solutions
$\overline{\phi}^{\mathbf p}$ and
$\underline{\phi}^{\mathbf p}$, respectively,
in $\mathcal{C}^2(\mathcal{V}_{\varepsilon})$. Moreover, for every $x\in[0,\varepsilon]$
and $i\in[I]$, we have:
\begin{equation}\label{eq: estime derivee seconde}
\partial_x^2\overline{\phi}^{\mathbf p}_i(x)
\geq 0,
\qquad
\partial_x^2\underline{\phi}^{\mathbf p}_i(x)
\leq 0.
\end{equation}
The following pointwise estimates hold:
\begin{equation}\label{eq: estime fonctions tests}
\begin{aligned}
|\overline{\phi}^{\mathbf p}_i(x)|
&\leq C_1,
&
|\underline{\phi}^{\mathbf p}_i(x)|
&\leq C_1,\\[1mm]
|\partial_x\overline{\phi}^{\mathbf p}_i(x)|
&\leq C_{2,i}^{u},
&
|\partial_x\underline{\phi}^{\mathbf p}_i(x)|
&\leq C_{2,i}^{v},\\[1mm]
|\partial_x^2\overline{\phi}^{\mathbf p}_i(x)|
&\leq C_{3,i}^{u},
&
|\partial_x^2\underline{\phi}^{\mathbf p}_i(x)|
&\leq C_{3,i}^{v}.
\end{aligned}
\end{equation}
Here,
\[
C_1=\frac{C_M+\eta}{\lambda},
\]
\[
\begin{aligned}
C_{2,i}^{u}
&=e^{\frac{C_M\varepsilon}{\sigma}}
\left(
\frac{|u_i(\varepsilon)-u(0)-\theta|}{\varepsilon}
+\frac{2(C_M+\eta)\varepsilon}{\sigma}
\right),\\[2mm]
C_{2,i}^{v}
&=e^{\frac{C_M\varepsilon}{\sigma}}
\left(
\frac{|v_i(\varepsilon)-v(0)+\theta|}{\varepsilon}
+\frac{2(C_M+\eta)\varepsilon}{\sigma}
\right),
\end{aligned}
\]
and
\[
\begin{aligned}
C_{3,i}^{u}
&=\frac{2(C_M+\eta)}{\sigma}
+\frac{C_M}{\sigma}C_{2,i}^{u},\\[2mm]
C_{3,i}^{v}
&=\frac{2(C_M+\eta)}{\sigma}
+\frac{C_M}{\sigma}C_{2,i}^{v}.
\end{aligned}
\]
\end{Lemma}
\begin{proof} 

Recall first that the existence and uniqueness of
$(\overline{\phi}^{\mathbf p},\underline{\phi}^{\mathbf p})$
follow from the standard existence and comparison principles
for semilinear elliptic ODEs.\\
We only establish the estimates for
$\overline{\phi}^{\mathbf p}$, the case of
$\underline{\phi}^{\mathbf p}$ being analogous. Since:
\[
\eta>\lambda(M+\theta),
\]
the constant functions:
\[
x\longmapsto \frac{C_M+\eta}{\lambda},
\qquad
x\longmapsto -\frac{C_M+\eta}{\lambda}
\]
are respectively a supersolution and a subsolution
of \eqref{eq : system EDO test sur ordre 2}. Moreover, by our choice of \(\eta\), these constant barriers dominate the prescribed boundary values.
Consequently, the comparison principle yields:
\[
|\overline{\phi}^{\mathbf p}_i(x)|
\leq \frac{C_M+\eta}{\lambda}
=:C_1,
\qquad \forall x\in[0,\varepsilon],\quad i\in[I].
\]
Furthermore, combining the auxiliary ODE with the preceding bound yields:
\[
\begin{aligned}
\sigma\partial_x^2\overline{\phi}^{\mathbf p}_i(x)
&=\lambda\overline{\phi}^{\mathbf p}_i(x)
+C_M\left(1+
|\partial_x\overline{\phi}^{\mathbf p}_i(x)|\right)+\eta\\
&\geq -\lambda C_1+C_M+\eta=0.
\end{aligned}
\]
Hence, $\overline{\phi}^{\mathbf p}_i$ is convex
on $[0,\varepsilon]$ for every $i\in[I]$. 
Next, using the auxiliary ODE and the previous estimate,
we obtain:
\[
0\leq\partial_x^2\overline{\phi}^{\mathbf p}_i(x)
\leq \frac{2(C_M+\eta)}{\sigma}
+\frac{C_M}{\sigma}
|\partial_x\overline{\phi}^{\mathbf p}_i(x)|.
\]
On the other hand, the boundary conditions give:
\[
\overline{\phi}^{\mathbf p}_i(\varepsilon)
-\overline{\phi}^{\mathbf p}_i(0)
=u_i(\varepsilon)-u(0)-\theta.
\]
By the mean value theorem, there exists
$\xi_i\in(0,\varepsilon)$ such that:
\[
\partial_x\overline{\phi}^{\mathbf p}_i(\xi_i)
=\frac{u_i(\varepsilon)-u(0)-\theta}{\varepsilon}.
\]
Applying Gronwall's inequality on both sides of
$\xi_i$, we deduce that, for every
$x\in[0,\varepsilon]$ and $i\in[I]$:
\[
|\partial_x\overline{\phi}^{\mathbf p}_i(x)|
\leq
e^{\frac{C_M\varepsilon}{\sigma}}
\left(
\frac{|u_i(\varepsilon)-u(0)-\theta|}{\varepsilon}
+\frac{2(C_M+\eta)\varepsilon}{\sigma}
\right)
\leq C_{2,i}^u,
\]
where,
\[
C_{2,i}^u:=
e^{\frac{C_M\varepsilon}{\sigma}}
\left(
\frac{\displaystyle
|u_i(\varepsilon)-u(0)-\theta|}{\varepsilon}
+\frac{2(C_M+\eta)\varepsilon}{\sigma}
\right).
\]
Finally, using once again the auxiliary ODE and
the gradient estimate obtained above, we deduce that
\[
0\leq
\partial_x^2\overline{\phi}^{\mathbf p}_i(x)
\leq C_{3,i}^u,
\]
where
\[
C_{3,i}^u=
\frac{2(C_M+\eta)}{\sigma}
+\frac{C_M}{\sigma}
e^{\frac{C_M\varepsilon}{\sigma}}
\left(
\frac{\displaystyle
|u_i(\varepsilon)-u(0)-\theta|}{\varepsilon}
+\frac{2(C_M+\eta)\varepsilon}{\sigma}
\right).
\]
These estimates hold for every
$x\in[0,\varepsilon]$ and $i\in[I]$. The corresponding estimates for \(\underline{\phi}^{\mathbf p}\) follow analogously, with \(v_i(\varepsilon)-v(0)+\theta\) replacing \(u_i(\varepsilon)-u(0)-\theta\), and with the opposite sign for the second derivative.
\end{proof}
The following proposition shows that the auxiliary functions $\overline{\phi}^{\mathbf p}$ and $\underline{\phi}^{\mathbf p}$, constructed above, serve as viscosity test functions for the supersolution $u$ and the subsolution $v$, respectively, at the junction point $\mathbf{0}$.
\begin{Proposition}\label{prop: optimal test functions}
Assume assumption $(\mathcal H)$ holds. Let
$u\in\mathcal L^b_{sc}(\mathcal N_R,\mathbb R)$
and
$v\in\mathcal U^b_{sc}(\mathcal N_R,\mathbb R)$
be, respectively, a generalized viscosity Kirchhoff supersolution and subsolution of
\eqref{eq : system second order}, in the sense of Definition \ref{def : weak sur/sous solutions}.\\
Let $M>0$ be chosen as in \eqref{eq strict born}.
For every $0<\varepsilon<\varepsilon_0$ and
$\theta>0$, choose $\eta>0$ sufficiently large
so that:
\[
\eta>\lambda(M+\theta),
\qquad
\mathbf p=(\varepsilon,\theta,\eta).
\]
Consider the associated auxiliary functions
$\overline{\phi}^{\mathbf p}$ and
$\underline{\phi}^{\mathbf p}$, defined by
\eqref{eq : system EDO test sur ordre 2}
and \eqref{eq : system EDO test sous ordre 2},
respectively.\\
Then $\overline{\phi}^{\mathbf p}$ and
$\underline{\phi}^{\mathbf p}$ are viscosity
test functions for $u$ and $v$, respectively,
at the junction point $\mathbf 0$.\\
In particular, they satisfy the strong Kirchhoff
inequalities:
\begin{equation}\label{eq Kirchoff uper phi}
\begin{aligned}
F\big(
u(0),
\partial_x\overline{\phi}^{\mathbf p}_1(0),
\ldots,
\partial_x\overline{\phi}^{\mathbf p}_I(0)
\big)&\leq 0,\\[2mm]
F\big(
v(0),
\partial_x\underline{\phi}^{\mathbf p}_1(0),
\ldots,
\partial_x\underline{\phi}^{\mathbf p}_I(0)
\big)&\geq 0.
\end{aligned}
\end{equation}
\end{Proposition}
\begin{proof}
We only prove the supersolution case, the subsolution case being analogous. Recall that the constant $M>0$ is a strict upper bound of $u$ in the neighborhood of $\bf 0$, namely: $\forall i \in [I],~~\forall x\in [0,\varepsilon_0],~u_i(x)<M$. \\
To obtain the result, it is enough to show that the minimum of $u-\overline{\phi}^{\mathbf{p}}$ is necessarily reached at the junction point $\bf 0$. Using the boundary conditions satisfied by $\overline{\phi}^{\mathbf{p}}$ in \eqref{eq : system EDO test sur ordre 2}, it follows first that $\forall i\in [I]$:
$$u(0)-\overline{\phi}^{\mathbf{p}}(0)-[u_i(\varepsilon)-\overline{\phi}^{\mathbf{p}}_i(\varepsilon)]=-\theta<0,$$
(recall $\theta>0$), implying that the minimum cannot be reached at the boundaries $(x=\varepsilon)_{i \in [I]}$.\\ Assume now that the minimum of $u-\overline{\phi}^{\mathbf{p}}$ is reached for some $y\in (0,\varepsilon)$, for a given $j\in [I]$, because  $u$ is a supersolution of \eqref{eq : system second order}, we obtain by definition:
\begin{eqnarray}\label{eq sur sol 0}
Q_j\big(y,u_j(y),\partial_x\overline{\phi}^{\mathbf{p}}_j(y),\partial_x^2\overline{\phi}^{\mathbf{p}}_j(y)\big)\ge 0.
\end{eqnarray}
On the other hand, using the strict ellipticity near $\bf 0$ and the strict growth w.r.t the unknown, $(\mathcal{H})$ (a) and $(\mathcal{H})$ (c), together with \eqref{eq: estime derivee seconde} ($\partial_x^2\overline{\phi}^{\mathbf{p}}\ge 0$), it follows that:
\begin{eqnarray*}
Q_j\big(y,u_j(y),\partial_x\overline{\phi}^{\mathbf{p}}_j(y),\partial_x^2\overline{\phi}^{\mathbf{p}}_j(y)\big) \leq \lambda (u_j(y)-M)+Q_j\big(y,M,\partial_x\overline{\phi}^{\mathbf{p}}_j(y),\partial_x^2\overline{\phi}_j^{\mathbf{p}}(y)\big)\\
\leq \lambda (u_j(y)-M)-\sigma\partial_x^2\overline{\phi}_j^{\mathbf{p}}(y)+Q_j\big(y,M,\partial_x\overline{\phi}^{\mathbf{p}}_j(y),0\big)
\end{eqnarray*}
Recall by \eqref{eq def CM}:
$$Q_j\big(y,M,\partial_x\overline{\phi}^{\mathbf{p}}_j(y),0\big)\leq C_M(1+|\partial_x\overline{\phi}^{\mathbf{p}}_j(y)|).$$
Observe also using the boundary conditions satisfied by $\overline{\phi}^{\mathbf{p}}$ in \eqref{eq : system EDO test sur ordre 2}, that the minimum of $u-\overline{\phi}^{\mathbf{p}}$ at $y$ is nonpositive, since:
$$u_j(y)-\overline{\phi}^{\mathbf{p}}_j(y)\leq u(0)-\overline{\phi}^{\mathbf{p}}(0)=0.$$
It follows therefore that:
\begin{eqnarray*}
&Q_j\big(y,u_j(y),\partial_x\overline{\phi}^{\mathbf{p}}_j(y),\partial_x^2\overline{\phi}_j^{\mathbf{p}}(y)\big) \leq \lambda (u_j(y)-\overline{\phi}^{\mathbf{p}}_j(y))+\lambda(\overline{\phi}^{\mathbf{p}}_j(y)-M) -\sigma\partial_x^2\overline{\phi}_j^{\mathbf{p}}(y)\\
&+C_M(1+|\partial_x\overline{\phi}^{\mathbf{p}}_j(y)|)
\leq \lambda\overline{\phi}^{\mathbf{p}}_j(y)-\lambda M -\sigma\partial_x^2\overline{\phi}_j^{\mathbf{p}}(y)+C_M(1+|\partial_x\overline{\phi}^{\mathbf{p}}_j(y)|).
\end{eqnarray*}
Using now the ODE satisfied by $\overline{\phi}^{\mathbf{p}}$ in \eqref{eq : system EDO test sur ordre 2}, we get:
\begin{eqnarray*}
\nonumber &Q_j\big(y,u_j(y),\partial_x\overline{\phi}^{\mathbf{p}}_j(y),\partial_x^2\overline{\phi}_j^{\mathbf{p}}(y)\big) \leq\\
&\lambda\overline{\phi}^{\mathbf{p}}_j(y)-\lambda M-\sigma\partial_x^2\overline{\phi}_j^{\mathbf{p}}(y)+C_M(1+|\partial_x\overline{\phi}^{\mathbf{p}}_j(y)|)=-\eta-\lambda M<0,\end{eqnarray*}
and then a contradiction with \eqref{eq sur sol 0} (recall $\eta>\lambda M$).\\
We are then able to conclude that the minimum of $u-\overline{\phi}^{\mathbf{p}}$ is reached at the junction point $\bf 0$. 
Since $u$ is a generalized Kirchhoff viscosity
supersolution and $\overline{\phi}^{\mathbf p}$
is a test function at the junction, we first obtain
\[
\max\left\{
\max_{i\in[I]}
Q_i\big(0,u(0),
\partial_x\overline{\phi}^{\mathbf p}_i(0),
\partial_x^2\overline{\phi}^{\mathbf p}_i(0)\big),
-F\big(u(0),
\partial_x\overline{\phi}^{\mathbf p}_1(0),
\ldots,
\partial_x\overline{\phi}^{\mathbf p}_I(0)\big)
\right\}\geq 0.
\]
However, by the preceding estimates, the
Hamiltonian terms are strictly negative:
\[
Q_i\big(0,u(0),
\partial_x\overline{\phi}^{\mathbf p}_i(0),
\partial_x^2\overline{\phi}^{\mathbf p}_i(0)\big)
<0,\qquad \forall i\in[I].
\]
Consequently, the generalized Kirchhoff condition
necessarily reduces to the strong Kirchhoff inequality:
\[
F\big(u(0),
\partial_x\overline{\phi}^{\mathbf p}_1(0),
\ldots,
\partial_x\overline{\phi}^{\mathbf p}_I(0)\big)
\leq 0.
\]
The same arguments apply for the subsolution $v$ to obtain:
$$F\big(v(0),\partial_x\underline{\phi}^{\mathbf{p}}_1(0),\ldots,\partial_x\underline{\phi}^{\mathbf{p}}(0)\big)\ge 0.$$
The proof is complete.
\end{proof}
\subsection{Generalized Kirchhoff viscosity solutions are Kirchhoff solutions}\label{subsec weak fort}
We now turn to the proof of Theorem \ref{th : weak are strong}, which establishes that every generalized viscosity Kirchhoff supersolution (resp. subsolution) in the sense of Definition \ref{def : weak sur/sous solutions} is a viscosity Kirchhoff supersolution (resp. subsolution) in the strong sense of Definition \ref{def : strong sur/sous solutions}. This result plays a crucial role in the existence theory based on Perron's method, where the generalized Kirchhoff condition naturally arises when taking semicontinuous envelopes and passing to limits, whereas our well-posedness result is formulated in terms of the strong Kirchhoff condition. To establish this equivalence, we adapt the auxiliary test functions constructed in Subsection \ref{subs ODE} by modifying their boundary conditions to incorporate an arbitrary viscosity test function at the junction, while preserving the same differential equations and their essential estimates.
\begin{proof}[Proof of Theorem \ref{th : weak are strong}]
We only prove the supersolution case, the subsolution case being analogous.\\
To obtain such a statement, let $u$ be a generalized solution of \eqref{eq PDE 0}, and $\psi$ a test function of $u$ at the junction point $\bf 0$. Our aim is to show that:
\begin{eqnarray}\label{eq: relation fondamentale}
F\big(u(0),\partial_x \psi_1(0),\ldots,\partial_x \psi(0)\big)\leq 0.   
\end{eqnarray}
There exists a neighborhood of the junction $\mathcal{V}_\varepsilon^\circ$ that we will denote for the sake of simplicity:
\[
\mathcal V_\varepsilon^\circ
=
\left(
\bigcup_{i=1}^{I}
([0,\varepsilon)\times\{i\})
\right)\big/\sim,
\qquad
(0,i)\sim(0,j),\quad \forall i,j\in[I].
\]
with $\varepsilon>0$ small enough, such that:
$$\underset{(x,i)\in \overline{\mathcal{V}_\varepsilon^\circ}}{\min}\{u_i(x)-\psi_i(x)\}=u(0)-\psi(0)=0.$$
Analogously to the construction in Subsection \ref{subs ODE}, we introduce an auxiliary function \(f\) whose components solve \begin{eqnarray} 
\label{eq : system EDO weak fort}
\begin{cases}
\lambda f_i(x)
-\sigma
\partial_x^2f_i(x)
+C_M\Big(
1+
\big|\partial_xf_i(x)\big|
\Big)
=-\eta,
&x\in(0,\varepsilon),
\\[1mm]
f(0)=u(0)=\psi(0),
\\[1mm]
f_i(\varepsilon)
=\psi_i(\varepsilon)-\theta,
 &i\in[I],
\end{cases}
\end{eqnarray}
The constant \(M>0\) is chosen as a strict upper bound for the absolute value of \(\psi\) in a fixed neighborhood of the junction, $|\psi|<M$, while $C_M,\lambda>0$ appear in assumptions $(\mathcal{H})$ and $\eta,\theta>0$ are parameters with $\eta>\lambda(M+\theta)$.\\
\textbf{Step 1:} $f$ is a test function of $u$ at $\bf 0$ and satisfies:
\begin{eqnarray}\label{eq: relation fondamentale 2}
F\big(u(0),\partial_x f_1(0),\ldots,\partial_x f(0)\big)\leq 0.   
\end{eqnarray}
Using the boundary conditions satisfied by $f$ in \eqref{eq : system EDO weak fort}, we have $\forall i\in [I]$:
$$0=u(0)-f(0)=u(0)-\psi(0)\leq u_i(\varepsilon)-\psi_i(\varepsilon)=u_i(\varepsilon)-f_i(\varepsilon)-\theta<u_i(\varepsilon)-f_i(\varepsilon)$$
(recall $\theta>0$). This implies that the minimum of $(u-f)$ cannot be reached at the boundaries $(\varepsilon)_{i \in [I]}$. \\
The same calculations as in Lemma \ref{lm estimee importantes} show that $f$ is convex $\partial_x^2f \ge 0$.
Assume now that the minimum of $u-f$ is reached for some $y\in (0,\varepsilon)$, for a given $j\in [I]$, since $u$ is a supersolution, we obtain:
\begin{eqnarray}\label{eq sur sol 1}
Q_j\big(y,u_j(y),\partial_xf_j(y),\partial_x^2f_j(y)\big)\ge 0.
\end{eqnarray}
Because $y$ is a minimum of $u-f$, we have: $u_j(y)-f_j(y)\leq u(0)-f(0)=0$, whereas the convexity of $f$ implies: $u_j(y)\leq f_j(y)\leq \max(f(0),f_j(\varepsilon))<M.$  
On the other hand, using the strict ellipticity near $\bf 0$, the strict growth w.r.t the unknown, assumptions $(\mathcal{H})$-(a) and (c), ($\partial_x^2f\ge 0$), it follows that:
\begin{eqnarray*}
Q_j\big(y,u_j(y),\partial_xf_j(y),\partial_x^2f_j(y)\big) \leq \lambda (u_j(y)-M)+Q_j\big(y,M,\partial_xf_j(y),\partial_x^2f_j(y)\big)\\
\leq \lambda (u_j(y)-M)-\sigma\partial_x^2f_j(y)+Q_j\big(y,M,\partial_xf_j(y),0\big)
\end{eqnarray*}
Recall from assumption $(\mathcal{H})$-(d):
$$Q_j\big(y,M,\partial_xf_j(y),0\big)\leq C_M(1+|\partial_xf_j(y)|).$$
It follows therefore that:
\begin{eqnarray*}
&Q_j\big(y,u_j(y),\partial_xf_j(y),\partial_x^2f_j(y)\big) \leq \lambda (u_j(y)-f_j(y))+\lambda(f_j(y)-M) -\sigma\partial_x^2f_j(y)\\
&+C_M(1+|\partial_xf_j(y)|)
\leq \lambda f_j(y)-\lambda M -\sigma\partial_x^2f(y)+C_M(1+|\partial_xf_j(y)|).
\end{eqnarray*}
Using now the ODE satisfied by $f$ in \eqref{eq : system EDO weak fort}, we get:
\begin{eqnarray*}
\nonumber &Q_j\big(y,u_j(y),\partial_xf_j(y),\partial_x^2f_j(y)\big) \leq\\
&\lambda f_j(y)-\lambda M-\sigma\partial_x^2f_j(y)+C_M(1+|\partial_xf_j(y)|)=-\eta-\lambda M<0,\end{eqnarray*}
and then a contradiction with \eqref{eq sur sol 1}.\\
We conclude that the minimum of $u-f$ is necessarily attained at $\bf 0$, and because for all $j\in[I]$:
$$Q_j\big(0,u(0),\partial_xf_j(0),\partial_x^2f_j(0)\big)<0,$$
(using the same last computations at $\bf 0$), it follows that $f$ is a test function of $u$ at $\bf 0$ and satisfies the Kirchhoff condition in the strong sense, namely \eqref{eq: relation fondamentale 2} holds true.\\
\textbf{Conclusion:} Observe now, with the same estimates  used in Lemma \ref{lm estimee importantes} (changing $u$ by $\psi$), one can show that:
$$
\sup_{x\in[0,\varepsilon]}
\left|
\partial_x^2
f_i(x)
\right| \leq C_1 +C_2 e^{C_2\varepsilon}
\left[ \frac{|\psi_i(\varepsilon)-\psi(0)-\theta|}{\varepsilon}+
\varepsilon C_1\right],
$$
for some constants positive $C_1,C_2$ independent of $(\varepsilon,\theta)$. Because $\psi$ is uniformly continuous on $\mathcal{V}_\varepsilon$, we get:
\begin{eqnarray}\label{eq: estim clef}
 \lim_{\varepsilon\searrow 0}\limsup_{\theta\searrow 0}~\Big[\varepsilon\sup_{x\in[0,\varepsilon]}
|\partial_x^2
f_i(x)|
\Big] =0.   
\end{eqnarray}
Using the expression of the boundary condition satisfied by $f$, we obtain that for all $i\in [I]$:
$$\psi_i(\varepsilon)-\psi(0)=f_i(\varepsilon)-f(0)+\theta.$$
Now, write:
\begin{eqnarray*}
f_i(\varepsilon)-f(0)
&=&
\varepsilon\partial_xf_i(0)
+
\int_0^\varepsilon\int_0^s
\partial_x^2f_i(t)\,dt\,ds.
\end{eqnarray*}
It follows therefore: 
$$\frac{\psi_i(\varepsilon)-\psi(0)}{\varepsilon}-\frac{\theta}{\varepsilon}-\frac{1}{\varepsilon}\int_0^\varepsilon\int_0^s
\partial_x^2f_i(t)\,dt\,ds=\partial_x f_i(0),~~\forall i\in[I]. $$
Using \eqref{eq: relation fondamentale 2}, we obtain:
\begin{eqnarray*}
&\ds F\Big(u(0),\frac{\psi_1(\varepsilon)-\psi(0)}{\varepsilon}-\frac{\theta}{\varepsilon}-\frac{1}{\varepsilon}\int_0^\varepsilon\int_0^s
\partial_x^2f_1(t)\,dt\,ds,\\
&\ds \ldots,\frac{\psi_I(\varepsilon)-\psi(0)}{\varepsilon}-\frac{\theta}{\varepsilon}-\frac{1}{\varepsilon}\int_0^\varepsilon\int_0^s
\partial_x^2f_I(t)\,dt\,ds\Big)=F\big(u(0),\partial_x f_1(0),\ldots,\partial_x f(0)) \leq 0.    
\end{eqnarray*}
We conclude using \eqref{eq: estim clef} together with the continuity of $F$ that: 
\begin{eqnarray*}
&F\big(u(0),\partial_x \psi_1(0),\ldots,\partial_x \psi(0)\big)=\\
&\ds\lim_{\varepsilon\searrow0}~~\underset{\theta\searrow 0}{\overline{\lim}}~~F\Big(u(0),\frac{\psi_1(\varepsilon)-\psi(0)}{\varepsilon}-\frac{\theta}{\varepsilon}-\frac{1}{\varepsilon}\int_0^\varepsilon\int_0^s
\partial_x^2f_1(t)\,dt\,ds,\\
&\ds \ldots,\frac{\psi_I(\varepsilon)-\psi(0)}{\varepsilon}-\frac{\theta}{\varepsilon}-\frac{1}{\varepsilon}\int_0^\varepsilon\int_0^s
\partial_x^2f_I(t)\,dt\,ds\Big)\leq 0,    
\end{eqnarray*}
namely \eqref{eq: relation fondamentale} holds true. The proof is complete.
\end{proof}
The equivalence established above allows us to work directly with the strong Kirchhoff condition in the subsequent comparison analysis, even when the solutions are initially understood in the generalized viscosity sense. In particular, the branchwise Hamiltonians no longer enter explicitly into the junction inequalities, which considerably simplifies the analysis of the coupling between the different equations on the network.
\subsection{Comparison Principle: Affine Perturbation and Junction Gradient Estimates} \label{subs comparison}
In this subsection, we establish the comparison principle stated in Theorem 2.4 for bounded semicontinuous viscosity subsolutions and supersolutions. Away from the junction, the comparison follows from the classical doubling-of-variables argument under Assumptions \((\mathcal{H})\)-(a) and \((\mathcal{H})\)-(b). To rule out a positive maximum at the junction, we combine the auxiliary test functions constructed in Subsection 3.1 with precise estimates of their junction gradients and a suitable affine perturbation of the supersolution. This perturbation preserves the branchwise supersolution property while making the Kirchhoff inequality strict at the junction, without requiring additional regularity of the solutions.\\
We begin with the following lemma, which provides the affine perturbation needed in the proof of Theorem \ref{th : compa syst}.
\begin{Lemma}[Affine perturbation of supersolutions]\label{lm affine pertur}
Assume assumption $(\mathcal{H})$ holds true. Let $u\in L^b_{sc}(\mathcal N_R,\mathbb R)$
be a bounded viscosity Kirchhoff
supersolution of \eqref{eq : system second order}. Then there exist $A>R$ and
$\delta_0>0$ such that, for every
$\delta\in(0,\delta_0]$, the function:
\[
u_i^\delta(x):=u_i(x)+\delta(A-x),
\qquad (x,i)\in\mathcal N_R,
\]
is a strict viscosity supersolution of the original
branchwise equations. Moreover, for every
$\varphi\in C^2(\mathcal N_R,\mathbb R)$ touching
$u^\delta$ from below at the junction:
\[
F\bigl(u^\delta(0),
\partial_x\varphi_1(0),\ldots,
\partial_x\varphi_I(0)\bigr)
\leq -\alpha I\delta,
\]
where $\alpha>0$ is the constant in $(\mathcal{H})$-(f).
\end{Lemma}
\begin{proof}
First choose \(M>\|u\|_\infty+1\), then \(A>R+L_M/\lambda\), and finally \(\delta_0>0\) sufficiently small so that \(\delta_0 A<1\).\\
Let $x_0\in(0,R)$, and suppose that
$\varphi_i\in C^2$ touches $u_i^\delta$
from below at $x_0$ for some $i\in [I]$. Then:
\[
\psi_i(x):=\varphi_i(x)-\delta(A-x)
\]
touches $u_i$ from below at $x_0$. Since $u$
is a viscosity supersolution,
\[
Q_i\bigl(x_0,u_i(x_0),
\partial_x\varphi_i(x_0)+\delta,
\partial_x^2\varphi_i(x_0)\bigr)\geq0.
\]
Using $(\mathcal{H})$-(a) and $(\mathcal{H})$-(d), we obtain:
\begin{eqnarray*}
Q_i\bigl(x_0,u_i^\delta(x_0),
\partial_x\varphi_i(x_0),\partial_x^2\varphi_i(x_0)\bigr)\ge Q_i\bigl(x_0,u_i(x_0),
\partial_x\varphi_i(x_0)+\delta,
\partial_x^2\varphi_i(x_0)\bigr)&\\
+\lambda\delta(A-x_0)-L_M\delta \ge \delta\bigl(\lambda(A-R)-L_M\bigr)>0.&
\end{eqnarray*}
Thus $u^\delta$ is a strict viscosity
supersolution of the original branchwise equations.\\
Now let $\varphi\in C^2(\mathcal N_R,\mathbb R)$
touch $u^\delta$ from below at the junction point $\bf 0$.
The function
\[
\psi_i(x):=\varphi_i(x)-\delta(A-x)
\]
touches $u$ from below at $\bf 0$. Because $u$ is a Kirchhoff supersolution:
\[
F\bigl(u(0),
\partial_x\varphi_1(0)+\delta,\ldots,\partial_x\varphi_I(0)+\delta\bigr)
\leq0.
\]
Since $u^\delta(0)=u(0)+\delta A$, assumptions
$(\mathcal{H})$-(e) and $(\mathcal{H})$-(f) yield
\[
\begin{aligned}
F\bigl(u^\delta(0),\partial_x\varphi_1(0),\ldots,\partial_x\varphi_I(0)\bigr)
&\leq F\bigl(u(0),\partial_x\varphi_1(0),\ldots,\partial_x\varphi_I(0)\bigr)\\
&\leq F\bigl(u(0),\partial_x\varphi_1(0)+\delta,\ldots,\partial_x\varphi_I(0)+\delta\bigr)
-\alpha I\delta\\
&\leq-\alpha I\delta.
\end{aligned}
\]
This completes the proof.
\end{proof}
We are now in a position to prove the comparison principle stated in Theorem \ref{th : compa syst}. The proof combines the affine perturbation established in Lemma \ref{lm affine pertur} with the auxiliary test functions constructed in Subsection \ref{subs ODE}. The key observation is that the nonnegative increments arising from the maximality condition can be absorbed into the leading gradient terms, while the remaining error terms are controlled using semicontinuity.
\begin{proof}[Proof of Theorem \ref{th : compa syst}]
Let $u$ and $v$ be, respectively, a bounded
lower semicontinuous viscosity supersolution
and a bounded upper semicontinuous viscosity
subsolution of \eqref{eq : system second order}. Let $A>0$ and $\delta_0>0$ be given by the preceding Lemma \ref{lm affine pertur} (with $M> \max(\|u\|_\infty,\|v\|_\infty)+1$). Then, for all $\delta \leq \delta_0$, the following map
\[
u_i^\delta(x)=u_i(x)+\delta(A-x),~~(x,i)\in \mathcal{N}_R
\]
is a viscosity supersolution of \eqref{eq : system second order} and satisfies the
strict Kirchhoff inequality:
\begin{eqnarray}\label{eq strict kirchoof}
F\bigl(u^\delta(0),\partial_x\varphi_1(0),\ldots,\partial_x\varphi_I(0)\bigr)
\leq-\alpha I\delta,
\end{eqnarray}
for every test function $\varphi$ touching
$u^\delta$ from below at the junction $\bf 0$.
Observe since by assumption: $v_i(R)\leq u_i(R)$, we have also $v_i(R)\leq u_i^\delta(R)$, for all $i\in [I]$.\\
We start then to establish a comparison theorem between the subsolution $v$ and the supersolution $u^\delta$ of \eqref{eq : system second order}. Since, the classical doubling-of-variables
argument, using $(\mathcal{H})$-(a) and-(b), excludes
a positive maximum in the interior
of any branch, with the aid of the last boundary conditions, our main target is to show that the following assumption:
\begin{eqnarray}\label{eq contra theo compa}
\sup_{(x,i)\in \mathcal N_R}\Big\{v_i(x)-u^\delta_i(x)\Big\}=v(0)-u^\delta(0)>0,    
\end{eqnarray}
will lead to a contradiction.\\
Fix then $\varepsilon>0$ and $\theta>0$ small enough, and  $\eta>0$ sufficiently large,
independently of $\varepsilon$ and $\theta$. We introduce 
\[
\Big(\overline\phi^{\mathbf{p},\delta},~~
\underline\phi^{\mathbf{p}}\Big),~~\mathbf{p}=(\varepsilon,\theta,\eta)
\]
the auxiliary functions constructed
in Subsection \ref{subs ODE}, by solving \eqref{eq : system EDO test sur ordre 2} and  \eqref{eq : system EDO test sous ordre 2}, respectively, with \(u\) replaced by \(u^\delta\) in \eqref{eq : system EDO test sur ordre 2}. Recall for clarity that the boundary conditions of $\overline\phi^{\mathbf{p},\delta}$ are: $\big(\overline\phi^{\mathbf{p},\delta}(0)=u^\delta(0),~
\overline\phi_i^{\mathbf{p},\delta}(\varepsilon)=
u_i^\delta(\varepsilon)-\theta\big)$ whereas the components of $\overline\phi^{\mathbf{p},\delta}$ solve the ODEs on the branches given in \eqref{eq : system EDO test sur ordre 2}.
It follows therefore by Proposition \ref{prop: optimal test functions}, applied to $u^\delta$
and $v$, these functions touch the
corresponding solutions at the junction.
The preceding Lemma \ref{lm affine pertur} and the strong
Kirchhoff condition therefore yield:
\begin{eqnarray}\label{hypoth contra kirchoff}
\nonumber F\bigl(u^\delta(0),
\partial_x\overline\phi_1^{\mathbf{p},\delta}(0),\ldots,\partial_x\overline\phi_I^{\mathbf{p},\delta}(0)\bigr)
\leq-\alpha I\delta,&\\
F\bigl(v(0),
\partial_x\underline\phi^{\mathbf{p}}_1(0),\ldots,\partial_x\underline\phi^{\mathbf{p}}_I(0)\bigr)\geq0&.
\end{eqnarray}
For each $i\in[I]$, 
from the integral representations
of the derivatives, we have:
\begin{eqnarray}\label{eq inte}
\nonumber \partial_x\overline\phi_i^{\mathbf{p},\delta}(0)-\partial_x\underline\phi_i^{\mathbf{p}}(0)=\frac{u_i^\delta(\varepsilon)-u^\delta(0)-[v_i(\varepsilon)-v(0)]-2\theta}{\varepsilon}\\
-\frac1\varepsilon
\int_0^\varepsilon\int_0^t
\partial_x^2\overline\phi_i^{\mathbf{p},\delta}(s)
\,ds\,dt+\frac1\varepsilon
\int_0^\varepsilon\int_0^t
\partial_x^2\underline\phi_i^{\mathbf{p}}(s)
\,ds\,dt
\end{eqnarray}
Lemma \ref{lm estimee importantes} provides a constant $C>0$,
independent of $\varepsilon$ and $\theta$,
such that:
\begin{eqnarray}\label{eq maj lapla}
\nonumber \Big|-\frac1\varepsilon
\int_0^\varepsilon\int_0^t
\partial_x^2\overline\phi_i^{\mathbf{p},\delta}(s)
\,ds\,dt+\frac1\varepsilon
\int_0^\varepsilon\int_0^t
\partial_x^2\underline\phi_i^{\mathbf{p}}(s)
\,ds\,dt\Big| \leq&\\
C\Big(\varepsilon+2\theta+|u_i^\delta(\varepsilon)-u^\delta(0)|+|v_i(\varepsilon)-v(0)|\big).
\end{eqnarray}
Now we use the following crucial identity:
\begin{eqnarray*}
|u_i^\delta(\varepsilon)-u^\delta(0)|+|v_i(\varepsilon)-v(0)|=(u_i^\delta(\varepsilon)-u^\delta(0))-(v_i(\varepsilon)-v(0))&\\
+2\max(u^\delta(0)-u_i^\delta(\varepsilon),0)+2\max(v_i(\varepsilon)-v(0),0)&.
\end{eqnarray*}
Combining this identity with \eqref{eq inte} and \eqref{eq maj lapla} we obtain:
\begin{eqnarray*}
\nonumber \partial_x\overline\phi_i^{\mathbf{p},\delta}(0)-\partial_x\underline\phi_i^{\mathbf{p}}(0) \ge&\\
\Big[(u_i^\delta(\varepsilon)-u^\delta(0))-(v_i(\varepsilon)-v(0))\Big](\frac{1}{\varepsilon}-C)-\frac{2\theta}{\varepsilon}&\\
C(\varepsilon+2\theta+2\max(u^\delta(0)-u_i^\delta(\varepsilon),0)+2\max(v_i(\varepsilon)-v(0),0)).&
\end{eqnarray*}
Recall that \eqref{eq contra theo compa} implies that:
$$(u_i^\delta(\varepsilon)-u^\delta(0))-(v_i(\varepsilon)-v(0))\ge 0.$$
Hence choosing $\varepsilon>0$ small enough such that:
$$\frac{1}{\varepsilon}-C>0,$$
one has:
\begin{eqnarray*}
\nonumber \partial_x\overline\phi_i^{\mathbf{p},\delta}(0)-\partial_x\underline\phi_i^{\mathbf{p}}(0)\ge&\\
-\frac{2\theta}{\varepsilon}-C(\varepsilon+2\theta+2\max(u^\delta(0)-u_i^\delta(\varepsilon),0)+2\max(v_i(\varepsilon)-v(0),0)):=-\mathcal{E}_i^{\varepsilon,\theta}.&
\end{eqnarray*}
We now use \eqref{hypoth contra kirchoff} together with assumption $(\mathcal{H})$-(e)-(f) and (g) to obtain (recall $u^\delta(0)<v(0)$):
\begin{eqnarray}
\nonumber 0\leq F\bigl(v(0),
\partial_x\underline\phi^{\mathbf{p}}_1(0),\ldots,\partial_x\underline\phi^{\mathbf{p}}_I(0)\bigr)
&\\ 
\nonumber \leq F\bigl(u^\delta(0),
\partial_x\overline\phi_1^{\mathbf{p},\delta}(0)+\mathcal{E}_1^{\varepsilon,\theta},\ldots,\partial_x\overline\phi_I^{\mathbf{p},\delta}(0)+\mathcal{E}_I^{\varepsilon,\theta}\bigr)&
\\
-\alpha I\delta+\omega_{F,M}(I\max_i\mathcal{E}_i^{\varepsilon,\theta}).
\end{eqnarray}
By the lower semicontinuity of $u^\delta$ and the upper semicontinuity of $v$ on the network, as established by the minimum and maximum conventions at the junction in Section \ref{sec : Notations}-\eqref{eq semi conti 0}, we have for every $i\in[I]$:
$$\lim_{\varepsilon \searrow 0}\max(u^\delta(0)-u_i^\delta(\varepsilon),0)=0,~~\lim_{\varepsilon \searrow 0}\max(v_i(\varepsilon)-v(0),0)=0,$$
implying that:
$$\lim_{\varepsilon\searrow0}\limsup_{\theta \searrow 0}~\omega_{F,M}(I\max_i\mathcal{E}_i^{\varepsilon,\theta})=0.$$
Hence $0\leq -\alpha I\delta$, which yields the desired contradiction. 
\\
Thus for the fixed $A>0$ and $\delta_0>0$, we have:
$$u(0)+\delta A \ge v(0),~~\forall \delta \leq \delta_0.$$
Sending $\delta \searrow 0$, we obtain $u(0)\ge v(0)$, and therefore $v\leq u$ on $\mathcal{N}_R$, as required.\\
The proof is complete.
\end{proof}
\begin{Remark}
A noteworthy feature of the preceding proof is that the construction and comparison of the auxiliary test-function envelopes at the junction rely only on the maximality of $v-u^\delta$, without requiring $v(0)-u^\delta(0)>0$. In particular, the strict positivity of this maximum is not involved in this part of the argument.
\end{Remark}
\subsection{Well-posedness}\label{subs existence Perron}
We now turn to the existence of viscosity solutions by Perron's method, thereby completing the well-posedness analysis. The preceding results provide the essential tools: the comparison principle established in Theorem 2.4 ensures uniqueness, while Theorem 2.3 shows that generalized Kirchhoff viscosity solutions satisfy the strong Kirchhoff condition. Combining these properties with Perron's construction, we establish the existence and uniqueness of a viscosity solution to problem \eqref{eq : system second order}.
\begin{proof}[Proof of Theorem \ref{th : existence}]
\textbf{Uniqueness:} Uniqueness follows directly from the comparison principle
established in Theorem \ref{th : compa syst}, whenever
the solutions agree at the outer endpoints.\\
\textbf{Existence:} We first construct a supersolution and a subsolution
of \eqref{eq : system second order} satisfying the
required Dirichlet inequalities at the endpoints
$(R,i)$, $i\in[I]$.\\
For this purpose, let $A,B>0$ be constants such that:
\begin{equation*}
\begin{aligned}
B\geq \frac{|F(0,0,\ldots,0)|}{I\alpha},&\\
K_B:=
\sup_{(x,i)\in\mathcal N_R}\left\{
|Q_i(x,0,-Be^{-x},Be^{-x})|,
|Q_i(x,0,Be^{-x},-Be^{-x})|
\right\},&\\
A\geq
\max\left\{
\frac{K_B}{\lambda},
\max_{i\in[I]}|a_i|
\right\},&
\end{aligned}
\end{equation*}
where $\alpha$ and $\lambda$ are the constants
appearing in Assumptions $(\mathcal H)$. Consider the function $f$ defined on each edge by:
\[
f_i(x)=A+Be^{-x},\qquad x\in[0,R],\quad i\in[I].
\]
By Assumption $(\mathcal H)$-(a), we obtain:
\[
\begin{aligned}
&Q_i(x,f_i(x),\partial_x f_i(x),
\partial_x^2 f_i(x)) \geq
Q_i(x,0,-Be^{-x},Be^{-x})
+\lambda(A+Be^{-x})\\
&\geq -K_B+\lambda A\geq0,
\qquad (x,i)\in(0,R)\times[I].
\end{aligned}
\]
Furthermore, Assumptions $(\mathcal H)$-(e)--(f)
yield:
\[
\begin{aligned}
F(f(0),\partial_x f_1(0),\ldots,
\partial_x f_I(0))
=F(A+B,-B,\ldots,-B)&\\
\leq F(0,-B,\ldots,-B)\leq F(0,0,\ldots,0)-I\alpha B\leq 0&.
\end{aligned}
\]
Consequently, $f$ is a classical Kirchhoff
supersolution of \eqref{eq : system second order},
satisfying:
\[
f_i(R)=A+Be^{-R}\geq a_i,
\qquad i\in[I].
\]
We now modify this global supersolution near the Dirichlet endpoints to enforce the boundary values exactly, while preserving the supersolution property. For this purpose, we introduce on each edge, the function:
\[
g_i(x)=a_i+K(R-x)-D(R-x)^2,
\qquad x\in[R-\delta,R],
\]
where the positive constants $K,D,\delta$ are
specified below. Let $\sigma_R>0$ and $\delta_R>0$ denote the uniform ellipticity constants near the Dirichlet
endpoints, as stated in the theorem. Set:
\[
d=A+B+\max_{i\in[I]}|a_i|+2,
\qquad
H=\max_{i\in[I]}|a_i|+2d+1.
\]
Let $L_H$ be the Lipschitz constant with respect
to the gradient variable for $|r|\leq H$, given
by Assumption $(\mathcal H)$, and define:
\[
C_H=
\max_{i\in[I]}
\sup_{\substack{x\in[R-\delta_R,R]\\ |r|\leq H}}
|Q_i(x,r,0,0)|.
\]
Choose $\delta\in(0,\delta_R)$ sufficiently small
so that
\[
L_H\delta<\sigma_R,
\qquad
\frac{(C_H+L_Hd/\delta+1)\delta^2}
{2(\sigma_R-L_H\delta)}<1.
\]
We then set:
\[
K=\frac{d}{\delta},
\qquad
D=\frac{C_H+L_HK+1}
{2(\sigma_R-L_H\delta)}.
\]
In particular, $D>0$ and $D\delta^2<1$. For every $x\in[R-\delta,R]$, we have
\[
\partial_xg_i(x)=-K+2D(R-x),
\qquad
\partial_x^2g_i(x)=-2D,
\]
and
\[
|g_i(x)|
\leq |a_i|+K\delta+D\delta^2
<|a_i|+d+1\leq H.
\]
Using the uniform ellipticity near $R$ and
the Lipschitz continuity with respect to the
gradient variable, we obtain:
\[
\begin{aligned}
&Q_i(x,g_i(x),\partial_xg_i(x),\partial_x^2g_i(x))\\
&\quad\geq
Q_i(x,g_i(x),0,0)
+2\sigma_R D-L_H|\partial_xg_i(x)|\\
&\quad\geq
-C_H+2\sigma_R D-L_H(K+2D\delta)\\
&\quad=
-C_H+2D(\sigma_R-L_H\delta)-L_HK=1>0.
\end{aligned}
\]
Thus $g_i$ is a classical strict supersolution
of the branchwise equation on $[R-\delta,R]$. Moreover, at the joining point $R-\delta$,
\[
\begin{aligned}
g_i(R-\delta)=a_i+K\delta-D\delta^2=a_i+d-D\delta^2&\\
>a_i+d-1>A+B\geq f_i(R-\delta).&
\end{aligned}
\]
Finally,
\[
g_i(R)=a_i<f_i(R)=A+Be^{-R}.
\]
Define now, for every $i\in[I]$,
\[
\overline W_i(x)=
\begin{cases}
A+Be^{-x}, & 0\leq x<R-\delta,\\
\min\{A+Be^{-x},g_i(x)\},
& R-\delta\leq x\leq R.
\end{cases}
\]
Since $g_i(R-\delta)>A+Be^{-(R-\delta)}$,
the two definitions agree in a neighborhood
of $R-\delta$. By the standard stability
properties of viscosity supersolutions under
finite minima, $\overline W$ is a lower
semicontinuous viscosity supersolution of
\eqref{eq : system second order}. Moreover,
it coincides with the exponential supersolution
near the junction, so the Kirchhoff condition
is preserved, and
\[
\overline W_i(R)=a_i,\qquad i\in[I].
\]
Similarly, considering
\[
\underline g_i(x)
=a_i-K(R-x)+D(R-x)^2,
\qquad x\in[R-\delta,R],
\]
and arguing as above, we obtain the upper
semicontinuous viscosity subsolution
\[
\underline W_i(x)=
\begin{cases}
-A-Be^{-x}, & 0\leq x<R-\delta,\\
\max\{-A-Be^{-x},\underline g_i(x)\},
& R-\delta\leq x\leq R,
\end{cases}
\]
which satisfies
\[
\underline W_i(R)=a_i,\qquad i\in[I].
\]
Furthermore, the comparison principle yields
$\underline W\leq\overline W$ on $\mathcal N_R$.
We can now apply the classical Perron method
between $\underline W$ and $\overline W$.
The upper and lower semicontinuous envelopes
$h^*$ and $h_*$ of the Perron function:
\[
h(x,i)=\sup\Big\{z_i(x):\ \underline W\leq z\leq\overline W,\ z\text{ generalized Kirchhoff viscosity subsolution}\Big\},
\]
are respectively, a generalized Kirchhoff viscosity
subsolution and supersolution of
\eqref{eq : system second order}, satisfying
the prescribed Dirichlet boundary conditions.
By Theorem \ref{th : weak are strong},
they satisfy the strong Kirchhoff condition
at the junction. Consequently, the comparison
principle of Theorem \ref{th : compa syst}
implies $h^*\leq h_*$. Since the reverse inequality $h_*\leq h^*$
always holds, we conclude that $h=h^*=h_*$. Therefore, $h$ is the unique continuous
viscosity solution of
\eqref{eq : system second order}, satisfying
the prescribed Dirichlet boundary conditions
and the strong Kirchhoff condition at the junction point. This completes the proof.
\end{proof}
\section{First-order junction problems: from generalized conditions to transmission and vanishing viscosity}\label{se 1 ordre et sto}
We conclude this work with a brief discussion of first-order junction problems, highlighting the limitations of generalized Kirchhoff conditions, the distinction between boundary and transmission mechanisms, and the role of vanishing viscosity in selecting an effective condition at the junction.\\
First, the generalized Kirchhoff condition does not in general imply the strong one.
Indeed, on the single branch $[0,1]$, take $H(r)=r$, $u\equiv0$, with $u(1)=0$, and let
$F$ satisfy $F(0)>0$. Since $u(0)=0$ and $H(u(0))=0$, the generalized supersolution
condition is satisfied:
\[
\max\{H(u(0)),-F(0)\}=\max\{0,-F(0)\}=0,
\]
whereas the strong Kirchhoff condition $F(0)\le0$ fails.\\
More recently, Forcadel and Monneau \cite{Forcadel Monneau 2026} considered first-order junction problems with continuous coercive, possibly non-convex, Hamiltonians $H_i$ and a continuous junction function $F_0$, non-increasing in each gradient variable. Setting
\[
H_{\min}(p)=\min_i H_i(p_i),\qquad H_{\max}(p)=\max_i H_i(p_i),
\]
they introduce
\[
\underline{\mathcal R}F_0(p)=\sup_{q\ge p}\min\{F_0(q),H_{\min}(q)\},\qquad
\overline{\mathcal R}F_0(p)=\inf_{q\le p}\max\{F_0(q),H_{\max}(q)\},
\]
and prove the non-trivial commutation identity
\[
\mathfrak RF_0:=\overline{\mathcal R}\,\underline{\mathcal R}F_0
=\underline{\mathcal R}\,\overline{\mathcal R}F_0.
\]
Thus the weak formulation associated with $F_0$ selects the effective junction law $\mathfrak RF_0$, which is imposed directly at the junction in the terminology of their strong formulation, but generally differs from the original condition $F_0$. They further characterize the same relaxation through Godunov fluxes and Riemann problems, with coercivity playing a central role in the construction. This relaxation mechanism has an important advantage: the min--max formulation is
naturally stable under limiting procedures. However, this robustness comes at a price
in comparison arguments. The Hamiltonian alternatives for the sub- and supersolution
may be realized on different branches, so that one is no longer comparing the same
Hamiltonian along a fixed edge. Instead, the argument may involve the largest mismatch
between the branchwise Hamiltonians. In this sense, the relaxation which ensures
stability also introduces the maximal inter-branch Hamiltonian error into the comparison
problem.\\
It is worth stressing that coercivity plays a central role in most first-order comparison results on junctions; see, in particular, \cite{Lions Souganidis 1,Lions Souganidis 2,control 5,Forcadel Monneau 2026}. In Lions and Souganidis \cite{Lions Souganidis 1,Lions Souganidis 2}, coercivity is used to obtain the Lipschitz control required in the comparison argument. After a sup/inf-convolution in time, they freeze the Hamiltonians at the junction and perform a blow-up around the contact point, reducing the problem to stationary one-dimensional profiles on each branch. The resulting comparison is then closed by a specific junction lemma, which itself only requires continuity of the branchwise Hamiltonians. In particular, the $H_i$ may differ from one edge to another: coercivity controls the admissible slopes, while the Kirchhoff condition couples the branches at the vertex.\\
The Lipschitz estimate also guarantees that the one-sided jets at the junction are non-empty.
However, their non-emptiness alone does not reduce the junction analysis to a standard
density argument: approaching the vertex through different branches leads, in general,
to different limiting Hamiltonians. Without such a gradient control, the situation may be
even worse, since the Kirchhoff condition itself may become vacuous. Indeed, on two branches set $Q_i(x,r,p)=r+G_i(xp^2)$, where
\[
G_1(s)=\begin{cases}-2,&s\le\frac14,\\ \frac83(s-1),&\frac14\le s\le1,\\0,&s\ge1,\end{cases}
\qquad
G_2(s)=\begin{cases}1,&s\le\frac14,\\ \frac53-\frac83s,&\frac14\le s\le1,\\-1,&s\ge1.\end{cases}
\]
For
\[
u=(2\sqrt{x},-\sqrt{x}),
\qquad
v=(1+\sqrt{x},1-2\sqrt{x}),
\]
one has \(u(1)=v(1)=(2,-1)\) and
\[
Q_1[u_1]=2\sqrt{x}\ge0,
\qquad
Q_2[u_2]=1-\sqrt{x}\ge0,
\]
whereas
\[
Q_1[v_1]=-1+\sqrt{x}\le0,
\qquad
Q_2[v_2]=-2\sqrt{x}\le0.
\]
Nevertheless,
\[
v(0)=1>0=u(0),
\]
while no \(C^1\) test touches \(u\) from below or \(v\) from above at the junction,
so the corresponding Kirchhoff conditions are vacuous and comparison fails.\\
In contrast, in the present elliptic framework the local non-degeneracy near the junction allows us to recover admissible jets directly. The auxiliary test functions $\bar\phi_i$ and $\underline\phi_i$ constructed above provide controlled one-sided slopes $\bar\phi_i'(0)$ and $\underline\phi_i'(0)$ at the vertex, which can then be adjusted by affine perturbation. More importantly, the comparison argument at the junction is reduced entirely to these slopes and to the Kirchhoff operator. No direct comparison between the branchwise Hamiltonians at $x=0$ is needed, so that the possible discontinuities of the $Q_i$ across the junction play no role in the final contradiction.\\
This is also consistent with the stochastic interpretation in terms of Walsh--spider diffusions. If $(x(t),i(t))$ denotes the radial position and the current branch, the non-stickiness property gives
\[
\int_0^t {\bf 1}_{\{x(s)=0\}}\,ds=0\qquad\text{a.s.}
\]
and It\^o's formula reads, up to the normalization of the local time,
\begin{eqnarray*}
f_{i(t)}(x(t))=f_{i(0)}(x(0))+M_t+&\\
\int_0^t\!\left[\frac{\sigma_{i(s)}^2(x(s))}{2}f_{i(s)}''(x(s))
+b_{i(s)}(x(s))f_{i(s)}'(x(s))\right]ds
+\Big(\sum_i\alpha_i f_i'(0)\Big)l_t .&    
\end{eqnarray*}
From a deterministic viewpoint, defining the integrand at $x=0$ is potentially problematic, since the branch index $i(s)$ is highly irregular there. Indeed, after the Walsh Brownian particle reaches the vertex, it returns to it infinitely often; more precisely, the set of visiting times is almost surely perfect and therefore uncountable. Nevertheless,
\[
\int_0^t {\bf 1}_{\{x(s)=0\}}\,ds=0 \qquad \text{a.s.},
\]
so that this irregularity is invisible to the ordinary time integral. Consequently, the precise value assigned to the branchwise coefficients at the vertex does not affect the $ds$-integral, whereas the transmission mechanism is entirely encoded by the local-time term and hence by the Kirchhoff condition.\\
These observations suggest that the junction should be viewed primarily as a transmission mechanism rather than as a classical boundary condition. In the stochastic picture, the passage through the vertex is governed by the branch-selection probabilities, while in the non-sticky elliptic regime the precise pointwise values of the branchwise coefficients at $x=0$ become irrelevant. This makes the elliptic Kirchhoff framework particularly natural. By contrast, in a sticky regime the process spends positive Lebesgue time at the junction, and one is then faced with the genuine difficulty of assigning meaningful values to the possibly discontinuous Hamiltonians and coefficients at $x=0$. This suggests that, for purely first-order problems, transmission conditions different from the classical Kirchhoff law may be more natural. In this perspective, it is natural to ask which junction condition is selected by a vanishing-viscosity approximation.\\
This discussion also suggests a possible transmission law arising from vanishing viscosity. Consider, for simplicity, a $C^2$ quasilinear approximation
\[
\varepsilon u_i^{\varepsilon\,\prime\prime}
   +H_i(x,u_i^\varepsilon,u_i^{\varepsilon\,\prime})=0,
\qquad
\sum_i\alpha_i u_i^{\varepsilon\,\prime}(0)=0 .
\]
Integrating the $i$-th equation over $(0,\varrho)$ and summing with the Kirchhoff weights $\alpha_i$ gives
\[
\frac1{\varrho}\sum_i\alpha_i\int_0^\varrho
H_i(x,u_i^\varepsilon,u_i^{\varepsilon\,\prime})\,dx
=
-\frac{\varepsilon}{\varrho}
\sum_i\alpha_i u_i^{\varepsilon\,\prime}(\varrho),
\]
since the contribution at $x=0$ cancels exactly by the Kirchhoff condition. Hence, if $\varrho=\varrho_\varepsilon\downarrow0$ is chosen so that
$\varepsilon/\varrho_\varepsilon\to0$, and suitable uniform gradient estimates allow the right-hand side to vanish and the averaged Hamiltonians to converge, one is naturally led to
\[
\sum_i\alpha_i H_i(0,u(0),p_i)=0,
\]
where $p_i$ denotes the limiting one-sided slope on the $i$-th branch. Thus the first-order limit may retain the transmission mechanism of the viscous problem not through the original Kirchhoff condition on the gradients, but through a Kirchhoff-type condition on the Hamiltonians themselves.\\
This viewpoint is also closely related to our previous probabilistic works
\cite{Spider,Spider 2,Ohavi control}, where the behavior at the vertex is encoded through the junction local time and the associated spinning or diffraction measure. Thus, in a first-order approximation, one should keep track not only of the vanishing diffusion, but also of the limiting occupation/local-time behavior at $x=0$; whether this contribution disappears or survives may determine the effective transmission law selected at the junction.

\end{document}